\documentclass{amsart}
\usepackage{hyperref}

\newcommand{\arxiv}[1]{\href{http://arxiv.org/abs/#1}{arXiv:#1}}
\newcommand*{\mailto}[1]{\href{mailto:#1}{\nolinkurl{#1}}}

\newtheorem{theorem}{Theorem}[section]
\newtheorem{lemma}[theorem]{Lemma}

\newcommand{\R}{{\mathbb R}}
\newcommand{\N}{{\mathbb N}}
\newcommand{\Z}{{\mathbb Z}}
\newcommand{\C}{{\mathbb C}}
\newcommand{\M}{{\mathbb M}}

\newcommand{\nn}{\nonumber}
\newcommand{\be}{\begin{equation}}
\newcommand{\ee}{\end{equation}}
\newcommand{\bea}{\begin{eqnarray}}
\newcommand{\eea}{\end{eqnarray}}

\newcommand{\ol}{\overline}
\newcommand{\ti}{\tilde}

\newcommand{\I}{\mathrm{i}}

\newcommand{\re}{\mathrm{Re}}

\newcommand{\Rg}[1]{P^{1/2}(#1)}

\newcommand{\eps}{\varepsilon}

\newcommand{\sig}{\sigma}
\newcommand{\lam}{\lambda}
\newcommand{\gam}{\gamma}
\newcommand{\om}{\omega}

\numberwithin{equation}{section}

\begin{document}

\title[Reconstruction of the Transmission Coefficient]{Reconstruction of the Transmission Coefficient for
Steplike Finite-Gap Backgrounds}

\author[I. Egorova]{Iryna Egorova}
\address{B.Verkin Institute for Low Temperature Physics\\ 47 Lenin Avenue\\61103 Kharkiv\\Ukraine}
\email{\mailto{egorova@ilt.kharkov.ua}}

\author[G. Teschl]{Gerald Teschl}
\address{Faculty of Mathematics\\
Nordbergstrasse 15\\ 1090 Wien\\ Austria\\ and International Erwin Schr\"odinger
Institute for Mathematical Physics, Boltzmanngasse 9\\ 1090 Wien\\ Austria}
\email{\mailto{Gerald.Teschl@univie.ac.at}}
\urladdr{\url{http://www.mat.univie.ac.at/~gerald/}}

\thanks{{\it Oper. Matrices {\bf 3}, 205--214 (2009)}}
\thanks{{\it Supported by Austrian Science Fund (FWF) under Grant No.\ Y330}}


\begin{abstract}
We consider scattering theory for one-dimensional Jacobi operators with respect to steplike
quasi-periodic finite-gap backgrounds and show how the transmission coefficient can be reconstructed from
minimal scattering data. This generalizes the Poisson--Jensen formula for the classical constant background
case.
\end{abstract}

\maketitle

\section{Introduction}

In classical one-dimensional scattering theory the transmission coefficient can be reconstructed from
the reflection coefficient via the well-known Poisson--Jensen formula. This formula plays a crucial
role in inverse scattering theory since it shows how to compute the left scattering data from the right one
and vice versa. Moreover, it is also one of the key ingredients for deriving the associated
sum rules  which have attracted an enormous amount of interest recently (see e.g.\ \cite{ks},
\cite{lns}, \cite{npy}, \cite{sizl}, \cite{zl}). Furthermore, these sum rules are intimately connected with
conserved quantities of the associated completely integrable lattices (see \cite{tist}, \cite{tjac}).
Finally, the reconstruction formula can be viewed as the solution of a scalar Riemann--Hilbert factorization
problem which arises in the nonlinear steepest decent method \cite{dz} when deriving
the long-time asymptotics (see \cite{km}, \cite{krt}, respectively \cite{krt2}).

Moreover, the same is true in case of scattering with respect to a finite-gap background \cite{emt}, \cite{voyu}.
In this situation the analogous formula was given in \cite{tag} including the associated sum rules (see also
\cite{emt2}, \cite{emt4}, \cite{mt}). Again there is a close relation with the solution of a scalar Riemann--Hilbert
factorization problem on the underlying Riemann surface which arises in the nonlinear steepest decent method
\cite{kt}, \cite{kt2}, \cite{kt3}, and \cite{krt3}.

However, while scattering theory with a steplike constant background is classical and goes back to the early sixties
\cite{bf} (see \cite{bet} for the most recent results), even in this case the reconstruction formula is unknown to the
best of our knowledge except for the case when the two spectra overlap. This might be related to the fact that the
case where the two spectra do not overlap will not be solved in terms of elementary methods, but will already
require tools from the theory of elliptic surfaces, as we will see below. In case of steplike finite-gap backgrounds,
scattering theory is again well-understood by now \cite{emt3}, \cite{emt5}, however, the reconstruction formula is
again unknown to the best of our knowledge (except for the case of two finite-gap backgrounds in the same isospectral
class \cite{emt3}). The main purpose of our paper is to fill this gap and provide a reconstruction formula for the left/right
transmission coefficient in terms of the left/right scattering data.

\section{Notation}
\label{sec:Not}

We begin by introducing some required background from the theory of hyperelliptic curves
to be used in the remainder of this article. For further information and proofs we refer for instance
to \cite{ak}, \cite{bght}, \cite{fk}, \cite{ghmt}, or \cite{tjac}.

Let $\M$ be the Riemann surface associated with the function $\Rg{z}$,
where
\begin{equation}
P(z) = \prod_{j=0}^{2g+1} (z-E_j), \qquad E_0 < E_1 < \cdots < E_{2g+1},
\end{equation}
$g\in \N$. $\M$ is a compact, hyperelliptic Riemann surface of genus $g$.
We will choose $\Rg{z}$ as the fixed branch
\begin{equation}
\Rg{z} = -\prod_{j=0}^{2g+1} \sqrt{z-E_j},
\end{equation}
where $\sqrt{.}$ is the standard root with branch cut along $(-\infty,0)$.

A point on $\M$ is denoted by $p = (z, \pm \Rg{z}) = (z, \pm)$, $z \in \C$.
The two points at infinity are denoted by $p = \infty_{\pm}$.
We use $\pi(p) = z$ for the projection onto the extended complex plane
$\C \cup \{\infty\}$.  The points $\{(E_{j}, 0), 0 \leq j \leq 2 g+1\} \subseteq \M$ are
called branch points and the sets
\begin{equation}
\Pi_{\pm} = \{ (z, \pm \Rg{z}) \mid z \in \C\backslash \Sigma\} \subset \M,
\qquad \Sigma= \bigcup_{j=0}^g[E_{2j}, E_{2j+1}],
\end{equation}
are called upper and lower sheet, respectively. Note that the boundary of
$\Pi_\pm$ consists of two copies of $\Sigma$ corresponding to the
two limits from the upper and lower half plane.

Let $\{a_j, b_j\}_{j=1}^g$ be loops on the Riemann surface $\M$ representing the
canonical generators of the fundamental group $\pi_1(\M)$. We require
$a_j$ to surround the points $E_{2j-1}$, $E_{2j}$ (thereby changing sheets
twice) and $b_j$ to surround $E_0$, $E_{2j-1}$ counter-clockwise on the
upper sheet, with pairwise intersection indices given by
\begin{equation}
a_j \circ a_k= b_j \circ b_k = 0, \qquad a_j \circ b_k = \delta_{jk},
\qquad 1 \leq j, k \leq g.
\end{equation}
The corresponding canonical basis $\{\zeta_j\}_{j=1}^g$ for the space of
holomorphic differentials can be constructed by
\begin{equation}
\underline{\zeta} = \sum_{j=1}^g \underline{c}(j)
\frac{\pi^{j-1}d\pi}{P^{1/2}},
\end{equation}
where the constants $\underline{c}(.)$ are given by
\[
c_j(k) = C_{jk}^{-1}, \qquad
C_{jk} = \int_{a_k} \frac{\pi^{j-1}d\pi}{P^{1/2}} =
2 \int_{E_{2k-1}}^{E_{2k}} \frac{z^{j-1}dz}{\Rg{z}} \in\R.
\]
The differentials fulfill
\begin{equation}
\int_{a_j} \zeta_k = \delta_{j,k}, \qquad \int_{b_j} \zeta_k = \tau_{j,k},
\qquad \tau_{j,k} = \tau_{k, j}, \qquad 1 \leq j, k \leq g.
\end{equation}
For further information we refer to \cite{fk}, \cite[App.~A]{tjac}.

In addition, we will need Green's function (in the potential theoretic sense) of the upper
sheet $\Pi_+$:

\begin{lemma}[\cite{tag}]
The Green function of $\Pi_+$ with pole at $z_0$ is given by
\be
g(z,z_0) = - \re \int_{E_0}^p \om_{p_0 \ti{p}_0}, \quad p=(z,+),\: p_0=(z_0,+),
\ee
where $\ti{p}_0= \ol{p_0}^*$ (i.e., the complex conjugate on the other sheet)
and  $\om_{p q}$ is the normalized Abelian differential of the third kind with poles
at $p$ and $q$.
\end{lemma}

Clearly, we can extend $g(z,z_0)$ to a holomorphic function on $\M\backslash\{p_0\}$
by dropping the real part. By abuse of notation we will denote this function by
$g(p,p_0)$ as well. However, note that $g(p,p_0)$ will be multivalued with
jumps in the imaginary part across $b$-cycles. We will choose the path
of integration in $\C\backslash[E_0,E_{2g+1}]$ to guarantee a single-valued
function.

From the Green function we obtain the Blaschke factor (cf.\ \cite{tag})
\be
B(p,\rho)= \exp \Big( g(p,\rho) \Big) = \exp\Big(\int_{E_0}^p \om_{\rho \rho^*}\Big),
\qquad \pi(\rho)\in\R,
\ee
which has the following properties:

\begin{lemma}\label{lem:blaschke}
The Blaschke factor satisfies
\be
B(E_0,\rho)=1, \quad\mbox{and}\quad
B(p^*,\rho) = B(p,\rho^*) = B(p,\rho)^{-1};
\ee
it is real-valued for $\pi(p)\in(-\infty,E_0)$.

Moreover,
\be
|B(p,\rho)|=1, \quad p \in \Sigma, \qquad
\arg(B(p,\rho)) = \delta_j(\rho), \qquad \pi(p)\in[E_{2j-1},E_{2j}],
\ee
where we set $E_{-1}=-\infty$, $E_{2g+2}=\infty$ and
\be
\delta_j(\rho) = \begin{cases}
0, & j=0,\\
\frac{1}{2} \int_{b_j} \om_{\rho \rho^*}, & j=1,\dots,g,\\
0, & j=g+1.
\end{cases}
\ee
\end{lemma}

\begin{proof}
The first part including the fact that $|B(p,\rho)|=1$, $p \in \Sigma$, is proven in \cite[Lem.~3.3]{tag}.
To see the formula for the argument first observe that $\om_{\rho \rho^*}$ is
real-valued on $\pi^{-1}(\R\backslash\Sigma)$ and purely imaginary on $\pi^{-1}(\Sigma)$. This
can be seen from the explicit expression (\ref{ompps}) for $\om_{\rho \rho^*}$ given below.
Hence, taking the path of integration along the lift of the real axis, we see that
the integral is real for $p=(\lam,\pm)$ with $\lam<E_0$ or $\lam>E_{2g+1}$. For
$p=(\lam,\pm)$ with $\lam\in[E_{2j-1},E_{2j}]$ the imaginary part is constant and given by
half the $b_j$ period.
\end{proof}

The above Abelian differential is explicitly given by
\be
\om_{p q} = \left( \frac{P^{1/2} + \Rg{p}}{2(\pi - \pi(p))} -
\frac{P^{1/2} + \Rg{q}}{2(\pi - \pi(q))} + P_{p q}(\pi) \right)
\frac{d\pi}{P^{1/2}},
\ee
where $P_{p q}(z)$ is a polynomial of degree $g-1$ which has to be determined from
the normalization $\int_{a_\ell} \om_{p p^*}=0$. In particular,
\be\label{ompps}
\om_{p p^*} = \left( \frac{\Rg{p}}{\pi - \pi(p)} + P_{p p^*}(\pi) \right)
\frac{d\pi}{P^{1/2}}.
\ee

\section{Reconstructing the transmission coefficient}

Let $H_q^\pm$ be two quasi-periodic finite-band Jacobi operators,\footnote{Everywhere in this
paper the sub or super index "$+$" (resp.\ "$-$") refers to the background on the right
(resp.\ left) half-axis.}
\be
H_q^\pm f(n) = a_q^\pm(n) f(n+1) + a_q^\pm(n-1) f(n-1) + b_q^\pm(n) f(n),
\quad f\in\ell^2(\Z),
\ee
associated with the hyperelliptic Riemann surface of the square root
\begin{equation}\label{defP}
P_\pm^{1/2}(z)= -\prod_{j=0}^{2g_\pm+1} \sqrt{z-E_j^\pm}, \qquad
E_0^\pm < E_1^\pm < \cdots < E_{2g_\pm+1}^\pm,
\end{equation}
where $g_\pm\in \N$ and $\sqrt{.}$ is the standard root with branch
cut along $(-\infty,0)$. In fact, $H_q^\pm$ are uniquely determined by
fixing a Dirichlet divisor $\sum_{j=1}^{g^\pm} \hat\mu_j^\pm$, where
$ \hat\mu_j^\pm=(\mu_j^\pm, \sig_j^\pm)$ with
$\mu_j^\pm\in[E_{2j-1}^\pm,E_{2j}^\pm]$ and $\sig_j^\pm\in\{-1, 1\}$.
The spectra of $H_q^\pm$ consist of $g_\pm+1$ bands
\be \label{0.1}
\sig_\pm:=\sig(H_q^\pm) = \bigcup_{j=0}^{g_\pm}
[E_{2j}^\pm,E_{2j+1}^\pm].
\ee
We are interested in scattering theory for the operator
\be \label{0.2}
H f(n)=a(n-1) f(n-1) + b(n) f(n)+ a(n) f(n+1), \quad n\in\Z,
\ee
whose coefficients are asymptotically close to the
coefficients of $H_q^\pm$ on the corresponding half-axes:
\be \label{hypo}
\sum_{n = 0}^{\pm \infty} |n| \Big(|a(n) -
a_q^\pm(n)| + |b(n) - b_q^\pm(n)| \Big) < \infty.
\ee
Let $\psi_q^\pm(z,n)$ be the Floquet solutions of the spectral equations
\be\label{0.12}
H_q^\pm \psi(n)=z \psi(n), \qquad z\in\C,
\ee
that decay for $z\in\C\setminus\sig_\pm$ as $n\to \pm\infty$.
They are uniquely defined by the condition $\psi_q^\pm(z,0)=1$,
$\psi_q^\pm(z,\cdot)\in\ell^2(\Z_\pm)$. The solution $\psi_q^+(z,n)$
(resp.\ $\psi_q^-(z,n)$)  coincides with the upper (resp.\ lower)
branch of the Baker--Akhiezer functions of $H_q^+$ (resp.\ $H_q^-$),
see \cite{tjac}.

The two solutions $\psi_\pm(z,n)$ of the spectral
equation
\be \label{0.11}
H \psi = z \psi, \quad z\in\C,
\ee
which are asymptotically close to the Floquet solutions $\psi_q^\pm(z,n)$ of the background
equations (\ref{0.12}) as $n\to\pm\infty$, are called Jost solutions.

Next, we introduce the sets
\be
\sig^{(2)}=\sig_+\cap\sig_-, \quad
\sig_\pm^{(1)}=\ol{\sig_\pm\setminus\sig^{(2)}}, \quad
\sig=\sig_+\cup\sig_-,
\ee
where $\sig$ is the (absolutely) continuous
spectrum of $H$ and $\sig_+^{(1)}\cup \sig_-^{(1)}$, $\sig^{(2)}$ are the parts which are of multiplicity
one, two, respectively.

In addition to the continuous part, $H$ has a finite number of
eigenvalues situated in the gaps, $\sig_d=\{\lam_1,...,\lam_s\} \subset \R \setminus \sig$
(see, e.g., \cite{tosc}).
For every eigenvalue we introduce the corresponding norming constants
\be
\gam_{\pm, k}^{-1}=\sum_{n \in\Z} |\psi_\pm(\lam_k, n)|^2, \quad 1 \leq k \leq s.
\ee
Note that this definition has to be slightly modified in the unlikely event that $\psi_q^\pm(z, n)$ and
hence $\psi_\pm(z, n)$ has a pole at $z=\lam_k$ (see \cite{emt5} for details).
The transmission and reflection coefficients are defined as usual via the scattering relations
\be
T_\mp(\lam) \psi_\pm(\lam,n) = \overline{\psi_\mp(\lam,n)} + R_\mp(\lam)\psi_\mp(\lam,n),
\quad\lam\in\sig_\mp.
\ee
Here the values of $\psi_\pm(\lam,n)$ for $\lam\in\sig_\pm$ are to be understood as limits from above
$\psi_\pm(\lam,n) = \lim_{\eps\downarrow 0} \psi_\pm(\lam+\I\eps,n)$ (the corresponding limits from
below just give the complex conjugate values $\ol{\psi_\pm(\lam,n)} = \lim_{\eps\downarrow 0}
\psi_\pm(\lam-\I\eps,n)$).

The following result is an immediate consequence of \cite[Lem.~5.1]{emt5}.

\begin{theorem}[\cite{emt5}]
Suppose $a(n)$, $b(n)$ satisfy \eqref{0.2}, then $a(n)$, $b(n)$ are uniquely determined by
one of the sets of its ``partial" scattering data ${\mathcal S}_+$ or ${\mathcal S}_-$, where
\begin{align}\nn
\mathcal{S}_\pm = \Big\{ & R_\pm(\lam),\, \lam\in\sig_\pm; \,
|T_\pm(\lam)|^2,\, \lam\in\sig_\mp^{(1)};\\
& \lam_1,\dots,\lam_s\in\R\setminus (\sig_+\cup\sig_-),\,
\gam_{\pm,1},\dots,\gam_{\pm,s}\in\R_+\Big\}.
\end{align}
\end{theorem}

This leads to the natural question if there is a simple way to compute $\mathcal{S}_+$
from $\mathcal{S}_-$ and vice versa (i.e., without solving the inverse scattering problem).
It turns out that this question reduces to the reconstruction of the transmission coefficient
$T_\pm(z)$ from $\mathcal{S}_\pm$. In fact, this follows from the following lemma.

\begin{lemma}[\cite{emt5}]
The transmission coefficients $T_\pm(z)$ admit a meromorphic extension to $\C\backslash\sig$.
In general they have simple poles at the eigenvalues $\lam_k$ of $H$.
In addition, there are simple poles at $\mu_j^\pm \in \R\backslash\sig_\pm$ which
are not poles of $\psi_q^\pm(z,1)$ (i.e., $\sig^\pm_j = \mp 1$) and simple
zeros at $\mu_j^\mp\in \R\backslash\sig_\mp$ which are poles of $\psi_q^\mp(z,1)$ (i.e., $\sig^\mp_j = \mp 1$).
A pole at $\mu_j^\pm$ could cancel with a zero at $\mu_j^\mp$ or could give a second order pole if
$\mu_j^\pm=\lam_k$.

Moreover, the entries of the scattering matrix have the following properties:
\[
\begin{array}{lcl}
{\bf (a)} & \rho_+(z) T_+(z) = \rho_-(z) T_-(z),\\
{\bf (b)}&\displaystyle{\frac{T_\pm(\lam)}{\overline{T_\pm(\lam)}}} = R_\pm(\lam),&
\lam\in\sig_\pm^{(1)},\nn\\
{\bf (c)}& 1 - |R_\pm(\lam)|^2 =
\displaystyle{\frac{\rho_\pm(\lam)}{\rho_\mp(\lam)}}\,|T_\pm(\lam)|^2,&\lam\in\sig^{(2)},\\
{\bf (d)}&\overline{R_\pm(\lam)}T_\pm(\lam) +
R_\mp(\lam)\overline{T_\pm(\lam)}=0,& \lam\in\sig^{(2)},
\end{array}
\]
where
\be
\rho_\pm(z)=\frac{\prod_{j=1}^{g_\pm}(z-\mu_j^\pm)}{ P_\pm^{1/2}(z)}.
\ee
\end{lemma}

Hence, the problem is to reconstruct the meromorphic function $T_+(z)$, $z\in\C\backslash\sig$ from its boundary
values
\be
\begin{cases}
|T_+(\lam)|^2, & \lam\in\sig_-^{(1)},\\
|T_+(\lam)|^2 = \frac{\rho_-(\lam)}{\rho_+(\lam)} (1 - |R_+(\lam)|^2), & \lam\in\sig^{(2)},\\
\displaystyle\frac{T_+(\lam)}{\overline{T_+(\lam)}} = R_+(\lam), & \lam\in \sig_+^{(1)}.
\end{cases}
\ee
That is, we know its absolute value on $\sig_-$ and its argument on the rest $\sig_+^{(1)}$. There will be
three Riemann surfaces involved, the one corresponding to $\sig=\sig_+\cup\sig_-$ and the ones corresponding to
$\sig_\pm$. All objects corresponding to $\sig$ will be denoted as in Section~\ref{sec:Not}, while the objects
associated with $\sig_\pm$ will have an additional $\pm$ sub/supscript.

\begin{theorem}
The transmission coefficient $T_+(z)$ can be reconstructed from the reflection
coefficient $R_+(z)$ and the eigenvalues $\lam_j$ via
\begin{align}\nn
T_+(z) = & \Bigg( \prod_{\mu_j^-\in M^-} B_-(z,\mu^-_j) \Bigg) \Bigg( \prod_{\mu_j^+\in M^+} B_-(z,\mu^+_j)^{-1} \Bigg)
\Bigg( \prod_{k=1}^s B_-(z,\lam_k)^{-1} \Bigg) \times\\
& \exp\Bigg(\frac{Q(z)^{-1}}{\pi \I} \int_{\sig_-^{(1)}}  Q \log(|T_+|) \om_{z z^*}\\ \nn
& + \frac{Q(z)^{-1}}{2\pi \I} \int_{\sig^{(2)}}  Q\, \big( \log\big(\frac{\rho_-}{\rho_+}\big) +\log (1 - |R_+|^2)\big) \om_{z z^*} \\
& + \frac{Q(z)^{-1}}{2\pi} \int_{\sig_+^{(1)}} Q\, \big(\arg(R_+) + \delta^-\big) \om_{z z^*} \Bigg),
\end{align}
where the integrals are taken over the lift of the indicated spectra to the upper sheet $\Pi_u$ (of the Riemann surface associated
with $\sig$). Moreover, we use the convention that we identify $z$ with $(z,+)$, and similarly for $\lam_k$, $\mu^\pm_j$, whenever used
in the argument of a function defined on a Riemann surface. Here
\be
M^\pm = \{\mu_j^\pm \,|\, \mu_j^\pm\in\R\setminus\sig \mbox{ and } \sig_j^\pm=-1 \},
\ee
\be
Q(z) = \prod_j \sqrt{z-e_j}, \quad\text{where $e_j$ are defined via}\quad \bigcup_j [e_{2j},e_{2j+1}] = \sig_+^{(1)},
\ee
and
\be
\delta^-(\lam)= \sum_\ell \delta^-_\ell \chi_{[E_{2\ell-1}^-,E_{2\ell}^-]}(\lam)
\ee
with (cf.\ Lemma~\ref{lem:blaschke})
\be
\delta^-_\ell = - \sum_{\mu_j^-\in M^-} \delta^-_\ell(\mu^-_j) + \sum_{\mu_j^+\in M^+} \delta^-_\ell(\mu^+_j) +
\sum_{k=1}^s \delta^-_\ell(\lam_k).
\ee
\end{theorem}

\begin{proof}
We start by considering the multivalued function
\be
t_+(z) =
\Bigg( \prod_{\mu_j^-\in M^-} B_-(z,\mu^-_j)^{-1} \Bigg) \Bigg( \prod_{\mu_j^+\in M^+} B_-(z,\mu^+_j) \Bigg)
\Bigg( \prod_{k=1}^s B_-(z,\lam_k) \Bigg) T_+(z)
\ee
which has neither zeros nor poles on $\Pi_u$ and satisfies
\be
\begin{cases}
|t_+(\lam)|^2 = |T_+(\lam)|^2, & \lam\in\sig_-,\\
\arg(t_+(\lam)) = \arg(T_+(\lam)) + \delta^-_\ell, & \lam\in\sig_+^{(1)} \cap [E_{2\ell-1}^-,E_{2\ell}^-].
\end{cases}
\ee
Moreover, the absolute value of $t_+(z)$ is single-valued and hence its logarithm is a
harmonic function on $\Pi_u$ which can be reconstructed from its boundary values.
To accommodate the fact that we know its absolute value on $\sig_-$ and its argument
on $\sig_+^{(1)}$ we consider
\be
Q(z) \log(t_+(z)).
\ee
Note that since $t_+(z)$ might still have zeros and poles on $\sigma$, the function $\log(t_+(z))$ might have
logarithmic singularities on $\sigma$.

Since $Q(\lam)$ is real-valued for $\lam\in\R\backslash \sig_+^{(1)}$ and purely imaginary for $\lam\in\sig_+^{(1)}$,
we infer that the real part of $Q(z) \log(t_+(z))$ is harmonic on $\Pi_u$ and can be reconstructed from
its boundary values
\be
\re\big(Q(\lam) \log(t_+(\lam))\big) = \begin{cases}
Q(\lam) \log(|T_+(\lam)|), & \lam\in\sig_-,\\
\I Q(\lam) \big(\arg(T_+(\lam)) + \delta^-_\ell\big), & \lam\in\sig_+^{(1)} \cap [E_{2\ell-1}^-,E_{2\ell}^-],
\end{cases}
\ee
using Green's function:
\begin{align}\nn
\re\big(Q(z) \log(t_+(z))\big) = &\re\Bigg(\frac{1}{\pi \I} \int_{\sig_-}  Q \log(|T_+|) \om_{z z^*}\\
&  + \frac{1}{\pi} \int_{\sig_+^{(1)}} Q \big(\arg(T_+) + \delta^-\big) \om_{z z^*} \Bigg).
\end{align}
Dropping the real part we get
\begin{align}\nn
T_+(z) = & \Bigg( \prod_{\mu_j^-\in M^-} B_-(z,\mu^-_j) \Bigg) \Bigg( \prod_{\mu_j^+\in M^+} B_-(z,\mu^+_j)^{-1} \Bigg)
\Bigg( \prod_{k=1}^s B_-(z,\lam_k)^{-1} \Bigg) \times\\
& \exp\Bigg(\frac{Q(z)^{-1}}{\pi \I} \int_{\sig_-}  Q \log(|T_+|) \om_{z z^*}
 + \frac{Q(z)^{-1}}{\pi} \int_{\sig_+^{(1)}} Q \big(\arg(T_+) + \delta^-\big) \om_{z z^*} \Bigg).
\end{align}
In fact, by \cite[Thm.~1]{voza} both the left-hand and the right-hand side have the same absolute value and hence can only differ by a constant
with absolute value one (in particular, the right-hand side is single-valued since the left-hand side is). This  constant must be
one since both sides are real-valued for real-valued $z$ to the left of $\sig$.
\end{proof}

\bigskip

\noindent{\bf Acknowledgments.} We want to thank Peter Yuditskii for helpful discussions.

\end{document}